\documentclass[reqno]{amsart}

\usepackage{amsmath}
\usepackage{amsthm}
\usepackage{amssymb}
\usepackage{mathrsfs}
\usepackage{enumerate}
\usepackage{color}
\usepackage[left=4cm,right=4cm]{geometry}
\usepackage[colorlinks=false, linkcolor=green, pdfborderstyle={/S/U/W 1}]{hyperref}
\numberwithin{equation}{section}

\theoremstyle{theorem}
\newtheorem{theorem}{Theorem}[section]

\newtheorem{proposition}[theorem]{Proposition}
\newtheorem{lemma}[theorem]{Lemma}
\theoremstyle{definition}

\theoremstyle{remark}
\theoremstyle{proof}

\newcommand{\R}{\mathbb{R}}

\renewcommand{\H}{\mathbb{H}}

\newcommand{\cE}{\mathcal{E}}
\newcommand{\cF}{\mathcal{F}}
\newcommand{\cG}{\mathcal{G}}

\newcommand{\cI}{\mathcal{I}}

\newcommand{\cP}{\mathcal{P}}

\newcommand{\cS}{\mathcal{S}}
\newcommand{\cT}{\mathcal{T}}

\newcommand{\sF}{\mathscr{F}}

\newcommand{\sT}{\mathcal{T}}

\newcommand{\dom}{\operatorname{Dom}}

\newcommand{\E}{\mathbb{E}}
\renewcommand{\P}{\mathbb{P}}

\newcommand{\ds}{\displaystyle}

\newcommand{\norm}[1]{\| #1 \|}

\newcommand{\paren}[1]{( #1 )}
\newcommand{\Paren}[1]{\left( #1 \right)}
\newcommand{\abs}[1]{| #1 |}
\newcommand{\Abs}[1]{\left| #1 \right|}
\newcommand{\brk}[1]{\langle #1 \rangle}

\renewcommand{\cE}{\mathbf{E}^{s}}
\renewcommand{\cP}{\mathbf{P}^{s}}
\newcommand{\proj}{\Phi}

\begin{document}

\title[Martingale transforms and the HLS inequality for semigroups]{Martingale transforms and the Hardy-Littlewood-Sobolev inequality for semigroups}
\author{Daesung Kim}
\address{Department of Mathematics, Purdue University, West Lafayette, IN 47907, USA}
\email{kim1636@purdue.edu}
\subjclass[]{}
\date{\today}
\keywords{}
\begin{abstract}
We give a representation of the fractional integral for symmetric Markovian semigroups as the projection of martingale transforms and prove the Hardy-Littlewood-Sobolev(HLS) inequality based on this representation. The proof rests on a new inequality for  a fractional Littlewood-Paley $g$-function.
\end{abstract}

\maketitle
\section{Introduction}

The classical inequality of Hardy, Littlewood \cite{HL28, HL30} and Sobolev \cite{So38} (HLS) has been extensively studied by many researchers for several years now. 
In particular, there has been a lot of effort to find the sharp constants of the HLS inequality.
In 1983, E. H. Lieb \cite{Li83} showed the existence of maximizing functions and the sharp constants by using symmetric decreasing rearrangements. 
E. A. Carlen and M. Loss \cite{CL90} derive the sharp HLS inequality in an ingenious way using the idea of competing symmetry.
Both proofs, however, utilize the symmetric decreasing rearrangement technique that relies quite heavily on the geometry of $\R^{d}$. 
In a recent paper \cite{FL12a}, R. Frank and E. Lieb employ a radically new, rearrangement-free method to compute the sharp constant for the HLS inequality on $\R^{d}$, which leads to an analogue of the sharp inequality on the Heisenberg group $\H^{d}$. 
 
The HLS inequality has been quite influential in applications to heat kernel estimates in many different settings since the pioneering work of N. Varopoulos in \cite{Va85}, E.~B.~Davis \cite{Da90} and others who put it in the frame of general Markovian semigroups; see also \cite{SalVarCou}.    The purpose of this paper is to give a probabilistic representation for fractional integrals for general symmetric Markovian semigroups and derive the HLS inequality based on the techniques of Gundy and Varopoulos \cite{GV79} used to represent Riesz transforms via harmonic extensions. 
Our representation is a variation of the one used by D. Applebaum and R. Ba\~nuleos in \cite{AB13} based on space-time Brownian motion often used for second order Riesz transforms.  
In \cite{AB13}, Applebaum and Ba\~nuelos give a proof of the HLS inequality on $\R^d$ using their representation and the martingale inequalities of Doob, Burkholder-Davis-Gundy.  
Unlike the space-time Brownian motion representation which requires a gradient in the space variable (or a carr\'e du champ), our representation only requires the time derivative which is well defined for general semigroups. 

The probabilistic representation of the fractional integrals can be thought of as martingale transforms where the predictable sequence is not bounded.  
Martingale transform techniques have been used quite effectively in the study of singular integral operators, particularly in obtaining optimal, or near optimal, inequalities.  
For some of this extensive literature on this subject, we refer the reader to \cite{ Ba86, BanMen, Ba13s, Ge10, Li08, Os3, VN04} and references therein. 
Given the powerful martingale and Bellman function methods pioneered by Burkholder in \cite{Bu84} to obtain sharp inequalities for martingale transform and their many subsequent uses in various problems in analysis and probability (see for example A. Os\c ekowski \cite{Os12}), it is natural to ask if those techniques can be extended to martingale transforms with unbounded multipliers and provide a different proof of the sharp HLS inequalities which could be extended to other settings. 
Unfortunately, as of now we have not been able to obtain the sharp results with the Bellman function methods.
This remains an interesting challenging problem.  

Our proof of HLS is based on the probabilistic representation for the fractional integral and relies on a new inequality for a fractional Littlewood-Paley $g$--function for general semigroups.  
For this, we will use the ``optimal'' splitting point technique of Stein \cite{St70s} and Hedberg \cite{He72} and an estimate for a classical Littlewood-Paley $g$--function in Stein \cite{St70t}. 
It is interesting to note that the boundedness of Stein's Littlewood-Paley $g$--function is based on the Burkholder-Davis-Gundy (BDG) inequalities for discrete martingales and hence, indirectly, the proof contained here for the HLS inequality is probabilistic.  
The basic question, in connection to the problem of finding the sharp inequality, is how to bypass the Littlewood-Paley $g$--function method.  
A much more preliminary and basic question is how to avoid the ``optimal'' splitting argument of  Stein \cite{St70s} and Hedberg \cite{He72}.  
This  optimal splitting is also a key step in the proof of Applebaum and Ba\~nuleos although it is done in combination with the BDG inequalities. 
It should also be mentioned here that a rather simple argument for general semigroups, which is essentially the same as the one given in \cite{St70s} for the classical case, is given in  \cite{AB13} without appealing to any probability. 
But again, those arguments use the ``optimal'' splitting and do not give the optimal constant.  

Let $\cS$ be a locally compact space with countable base equipped with a positive Radon measure $dx$ on $\cS$ and $\{T_{t}\}_{t\geq0}$ a strongly continuous symmetric Markovian semigroup.
Furthermore, we assume that the semigroup is Feller and has the \emph{Varopoulos} dimension $d$ that we shall define below.
The fractional integral of order $\alpha$ ($0<\alpha<d$) associated to $\{T_{t}\}_{t\geq0}$ is defined by
\begin{eqnarray}\label{eq:fracint}
\cI_{\alpha}(f)(x)
=\frac{1}{\Gamma(\frac{\alpha}{2})}\int_{0}^{\infty}t^{\frac{\alpha}{2}-1}T_{t}f(x)dt.
\end{eqnarray}
It is noteworthy that if $\{T_{t}\}_{t\geq0}$ is the standard heat semigroup on $\R^{d}$ then the definition of the fractional integral \eqref{eq:fracint} is equivalent to
\begin{eqnarray*}
\cI_{\alpha}(f)(x)
=\frac{\Gamma(\frac{d-\alpha}{2})}{2^{\alpha}\pi^{d/2}\Gamma(\frac{\alpha}{2})}\int_{\R^{d}}\frac{f(y)}{|x-y|^{d-\alpha}}dy.
\end{eqnarray*}
Let $\frac{1}{q}=\frac{1}{p}-\frac{\alpha}{d}$, $\frac{1}{q}+\frac{1}{q'}=1$ and $0<\alpha<d$, then the HLS inequality for $\cI_{\alpha}$ states that there exists a constant $C_{\alpha,p,d}$ such that
\begin{eqnarray}\label{eq:HLS}
\abs{\brk{\cI_{\alpha}(f),h}}\leq C_{\alpha,p,d}\norm{f}_{p}\norm{h}_{q'}
\end{eqnarray}
for every $f\in L^{p}$ and $h\in L^{q'}$.

Suppose that $(X_{t})_{t\geq0}$ is a stochastic process associated to $\{T_{t}\}_{t\geq0}$ and $(Y_{t})_{t\geq0}$ is a standard 1-dimensional Brownian motion independent of $(X_{t})_{t\geq0}$. 
We shall denote by  $Z_{t}=(X_{t}, Y_{t})$ for simplicity.
We note that the stochastic process $(X_{t})_{t\geq0}$ has a c\`adl\`ag version and the strong Markov property due to the Feller property of $\{T_{t}\}_{t\geq0}$.
We assume that the initial distribution of $(Z_{t})_{t\geq0}$ is given by $dx\otimes\delta_{s}$ for fixed $s>0$ and denote by $\cE$ the corresponding expectation. 
Let $\tau$ be the hitting time of $Y_{t}$ at $0$ and denote by $\{P_{y}\}_{y\geq0}$ the Poisson semigroup associated with $\{T_{t}\}_{t\geq0}$.  (We will describe this semigroup more precisely below.) Let $u_{f}(x,y)=P_yf(x)$ be the harmonic extension of $f$ defined on $\cS\times [0,\infty)$.
We set 
\begin{eqnarray}\label{probrep}
\sT_{\alpha}^{s}(f)(x)=\cE[\int_{0}^{\tau}Y_{t}^{\alpha}\frac{\partial u_{f}}{\partial y}(Z_{t})dY_{t}|X_{\tau}=x].
\end{eqnarray}
The main result of the paper is to show that $\sT_{\alpha}^{s}$ defined above gives a probabilistic representation of the fractional integral and that it satisfies the analogue of the HLS inequality \eqref{eq:HLS}.  More precisely we ahve  
\begin{theorem}\label{thm1}
Let $s>0$ and $f,h\in C_{0}(\cS)$. If $\frac{1}{q}=\frac{1}{p}-\frac{\alpha}{d}$, $1<p<q<\infty$ and $0<\alpha<d$, then we have 
\begin{eqnarray}\label{eq:thm1}
\Abs{\langle \sT^{s}_{\alpha}f,h \rangle }
=\Abs{ \cE[\int_{0}^{\tau}Y_{t}^{\alpha}\frac{\partial u_{f}}{\partial y}(Z_{t})\frac{\partial u_{h}}{\partial y}(Z_{t})dt]}
\leq C_{\alpha,p,d}\norm{f}_{p}\norm{h}_{q'}
\end{eqnarray}
where $C_{\alpha,p,d}$ depends only on $\alpha$, $p$ and $d$ and $q'$ is the conjugate exponent of $q$. 
As a consequence, we obtain
\begin{eqnarray*}
\lim_{s\to\infty}\sT_{\alpha}^{s}(f)=\frac{\Gamma(\alpha+2)}{2^{\alpha+2}} \cI_{\alpha}(f)
\end{eqnarray*}
in the sense of distributions.
\end{theorem}
The proof of Theorem \ref{thm1} relies on an auxiliary function which satisfies an HLS-type inequality. To be specific, we define the fractional Littlewood-Paley function $\cG_{\alpha}$ by
\begin{eqnarray}\label{eq:fracgfunc}
\cG_{\alpha}(f)(x)=\Paren{\int_{0}^{\infty}y^{2\alpha+1}\Abs{\frac{\partial u_{f}}{\partial y}(x,y)}^{2}dy}^{1/2}.
\end{eqnarray}
The next theorem says that the fractional Littlewood-Paley function enjoys the HLS-type inequality, which will help us to show that the probabilistic representation $\sT_{\alpha}^{s}$ inherits the HLS inequality from $\cG_{\alpha}$.

\begin{theorem}\label{lem1}  
Let $\frac{1}{q}=\frac{1}{p}-\frac{\alpha}{d}>0$, $1<p<q<\infty$ and $0<\alpha<d$. If $f\in L^{p}(\cS)$, then the fractional Littlewood-Paley function $\cG_{\alpha}(f)$ defined in \eqref{eq:fracgfunc} satisfies
\begin{eqnarray*}
\norm{\cG_{\alpha}(f)}_{q}\leq C_{\alpha,p,d}\norm{f}_{p}.
\end{eqnarray*}
\end{theorem}

The rest of the paper is organized as follows. 
In \S\ref{sec:semigp}, we collect the required definitions and results from the general heat semigroup theory which allow us to carry out the Gundy-Varopoulos construction.
In \S\ref{sec:stocproc}, we review some of the tools of Varopoulos \cite{Va80}.
More precisely, we construct the stochastic process corresponding to the semigroup and introduce basic concepts such as the initial distribution, stochastic integrals and conditional expectations. 
We finish this section by presenting the projection theorem that plays a critical role in the proof of Theorems \ref{thm1}.  
The last section is devoted to the proofs of Theorems \ref{thm1} and \ref{lem1}.

\subsection*{Notations}
The space of all continuous functions on $\cS$ vanishing at $\infty$ is denoted by $C_{0}(\cS)$. 
We also use $C_{c}(\cS)$ to denote the space of all compactly supported continuous functions.
The lower case letter$c$, $c_{1}$, $c_{2} \cdots$ denote generic constants which may change from line to line.
We use the notation $C_{p,q,r}$ to specify that the constant depends on $p$, $q$ and $r$.
We denote the inner product by $\brk{f,g}=\int_{\cS}f(x)g(x)dx$ for notational convenience.
The domain of an operator $A$ is denoted by $\dom(A)$.

\section{Preliminaries}
\subsection{General semigroup theory}\label{sec:semigp}  
Here we recall some facts about semigroups concentrating on what we need in the subsequent sections, particularly the definition of strongly continuous symmetric Markovian semigroup and the construction of the Poisson semigroup used in the probabilistic representation of the fractional integral \eqref{probrep}. 

We say that a semigroup $\{T_{t}\}_{t\geq0}$ on $\cS$ is a \emph{symmetric Markovian semigroup} if it satisfies the followings properties. 
\begin{enumerate}
\item[(S1)] $T_{t}f\geq 0$ whenever $f\geq 0$;
\item[(S2)] $T_{t}1=1$;
\item[(S3)] (Symmetry) $\langle T_{t}f, g \rangle=\langle f, T_{t}g \rangle$ for every $f,g\in L^{2}(\cS)$ and every $t\geq 0$;
\item[(S4)] ($L^{p}$-contraction) $\norm{T_{t}f}_{p}\leq \norm{f}_{p}$ whenever $f\in L^{p}(\cS)$ for every $1\leq p\leq\infty$.
\end{enumerate}
Suppose that there exists a symmetric Markovian semigroup $\{T_{t}\}_{t\geq0}$ on $\cS$.
Furthermore, we assume that the semigroup is \emph{strongly continuous} on $L^{2}(\cS)$ and a \emph{Feller semigroup}. In other words, $\{T_{t}\}_{t\geq0}$ satisfies for any $f\in C_{0}(\cS)$ that
\begin{enumerate}
\item[(S5)] (Strong continuity) $\ds\lim_{t\to0}\norm{T_{t}f-f}_{2}=0$ for all $f\in L^{2}(\cS)$ and 
\item[(S6)] (Feller) for all $f\in C_{0}(\cS)$, $T_{t}f\in C_{0}(\cS)$ for all $t\geq0$ and $\ds\lim_{t\to0}\norm{T_{t}f-f}_{\infty}=0$.
\end{enumerate}
In \S\ref{sec:stocproc}, we construct a stochastic process associated to the semigroup $\{T_{t}\}_{t\geq0}$. The assumption that the semigroup is Feller ensures that the stochastic process has c\`adl\`ag paths and the strong Markov property.
The semigroup $\{T_{t}\}_{t\geq0}$ is assumed to have the \emph{Varopoulos} dimension $d$ ($d>2$) introduced by Varopoulos in \cite{Va85}, meaning that
\begin{enumerate}
\item[(S7)](Varopoulos dimension) for all $f\in L^{p}(\cS)$, $1\leq p<\infty$ and $t>0$, there is a constant $c>0$ such that
\begin{eqnarray*}
\norm{T_{t}f}_{\infty}\leq c t^{-\frac{d}{2p}}\norm{f}_{p}.
\end{eqnarray*}
\end{enumerate}
For instance, the heat semigroup $e^{-t\Delta}$ on $\R^{d}$($d\geq 3$) has the Varopoulos dimension $d$. 

Given the symmetric Markovian semigroup, one can define the Poisson semigroup associated to $\{T_{t}\}_{t\geq0}$ in the following ways.
One can define the Poisson semigroup by means of the spectral decomposition on $L^{2}(\cS)$. 
For $f\in L^{2}(\cS)$, the semigroup $\{T_{t}\}_{t\geq0}$ can be written as
\begin{eqnarray*}
T_{t}f(x)=\int_{0}^{\infty}e^{-\lambda t}dE_{\lambda}f(x).
\end{eqnarray*}
Here, the family $\{E_{\lambda}:\lambda\geq0\}$ is the spectral resolution associated to the infinitesimal generator of $T_{t}$. 
Then the Poisson semigroup associated to the semigroup $\{T_{t}\}$ on $L^{2}$ is defined by
\begin{eqnarray*}
P_{t}f(x)=\int_{0}^{\infty}e^{-\lambda^{1/2} t}dE_{\lambda}f(x).
\end{eqnarray*}
The other way to introduce the Poisson semigroup is to subordinate the given semigroup $\{T_{t}\}_{t\geq0}$ (see S. Bochner \cite{Bo12}), which enables us to define the Poisson semigroup on $L^{p}$ for any $1\leq p\leq\infty$. 
To be specific, for $f\in L^{p}$, we define the Poisson semigroup by
\begin{eqnarray}\label{eq:bocpt}
P_{t}f(x)=\int_{0}^{\infty}T_{s}f(x)\mu_{t}(ds)
\end{eqnarray}
where $\mu_{t}(ds)=\frac{t}{2\sqrt{\pi}}e^{-t^{2}/4s}s^{-3/2}ds$ and $1\leq p\leq \infty$. 
One can easily verify by direct calculation that these two construction coincides when $p=2$. 
We notice that the construction of the Poisson semigroup by $\mu_{t}(ds)$ is a special case of the subordination.
Generally speaking, one obtains a new semigroup by subordinating with any convolution measure on $[0,\infty)$ (a L\'evy process on $[0,\infty)$ from a probabilistic point of view).
In \eqref{eq:bocpt}, we adopt the convolution measure $\mu_{t}(ds)$ called the $\frac{1}{2}$-stable subordinator.
From now on, we call $u_{f}(x,y):=P_{y}f(x)$ \emph{the harmonic extension} of $f$.

The following lemma tells us that the Poisson semigroup satisfies the same properties as $\{T_{t}\}_{t\geq0}$ to some extent. 
\begin{lemma}\label{lem:sms}
Let $\{T_{t}\}_{t\geq0}$ be a strongly continuous symmetric Markovian semigroup and $\{P_{t}\}_{t\geq0}$ the corresponding Poisson semigroup. 
Then the semigroup $\{P_{t}\}_{t\geq0}$ is also a strongly continuous symmetric Markovian semigroup. 
In addition, if the semigroup $\{T_{t}\}_{t\geq0}$ has the Varopoulos dimension $d$, then the dimension of the Poisson semigroup $\{P_{t}\}_{t\geq0}$ is $2d$, namely
\begin{eqnarray*}
\norm{P_{y}f}_{\infty}=\norm{u_{f}(\cdot,y)}_{\infty} \leq \frac{c}{y^{d/p}}\norm{f}_{p}.
\end{eqnarray*}
\end{lemma}
\begin{proof}
Let us firstly show that the Poisson semigroup is a strongly continuous symmetric Markovian semigroup. The assumptions (S1), (S2) and (S3) follow directly from the definition \eqref{eq:bocpt}. To see the $L^{p}$-contraction of $\{P_{y}\}_{y\geq0}$, we make use of the definition, Fubini's theorem and Jensen's inequality that
	\begin{eqnarray*}
	\norm{P_{y}f}_{p}^{p}
	&=&\int_{\cS}\abs{P_{y}f(x)}^{p}dx
	\leq\int_{\cS}\int_{0}^{\infty}\abs{T_{s}f(x)}^{p}\mu_{y}(ds)dx\\
	&=&\int_{0}^{\infty}\norm{T_{s}f}_{p}^{p}\mu_{y}(ds)
	\leq\norm{f}_{p}^{p}.
	\end{eqnarray*}
In the same way, one can show that the Poisson semigroup is strongly continuous on $L^{2}$. 
We see by using the assumption that $\{P_{y}\}_{y\geq0}$ has the dimension $d$ that
	\begin{eqnarray*}
	\abs{P_{y}f(x)}
	&=&\Abs{\int_{0}^{\infty}T_{s}f(x)\mu_{y}(ds)}
	\leq \int_{0}^{\infty}\abs{T_{s}f(x)}\mu_{y}(ds)\\
	&\leq&c\norm{f}_{p}\int_{0}^{\infty}s^{-\frac{d}{2p}}\mu_{y}(ds).
	\end{eqnarray*}
The direct calculation yields the integral in the last term is equal to $Cy^{-\frac{d}{p}}$, which implies that the Poisson semigroup has the dimension $2d$.
\end{proof}

It is well-known that the map $y\mapsto u_{f}(\cdot,y)$ is real-analytic for all $f\in L^{p}$, $1<p<\infty$. (See \cite[p.67, p.72]{St70t}.)
This observation is important in the present context because the representation of fractional integrals requires that the harmonic extension $u_{f}(x,y)$ in the general setting is differentiable with respect to $y$. 
Next lemma is concerned with a derivative estimate for the harmonic extension $u_{f}$. 
\begin{lemma}\label{lem:ptest}
Let $f$ be a bounded measurable function on $\cS$. We have the following estimate for the harmonic extension of $f$. 
\begin{eqnarray*}
\Abs{y\frac{\partial u_{f}}{\partial y}(x,y)}\leq c_{1}u_{|f|}(x,\frac{y}{\sqrt{2}})
\end{eqnarray*}
for some constant $c_{1}$ independent of $f$.
\end{lemma}

\begin{proof}
It is verified by using Bochner's subordination. 
To be specific, if we write $\mu_{y}(ds)=\frac{1}{2\sqrt{\pi}}\eta_{y}(s)ds$, we have
\begin{eqnarray*}
y\frac{\partial\eta_{y}(s)}{\partial y}=(1-\frac{y^{2}}{2s})ye^{-y^{2}/4s}s^{-3/2}.
\end{eqnarray*}
Note that there exists a constant $c_{1}$ such that $\abs{1-\frac{y^{2}}{2s}}\leq c_{1}e^{y^{2}/8s}$ for every $y>0$ and $s>0$. It then follows that
\begin{eqnarray*}
\Abs{y\frac{\partial \eta_{y}}{\partial y}(s)}\leq c_{1}ye^{-y^{2}/8s}s^{-3/2}=c_{1}\eta_{\frac{y}{\sqrt{2}}}(s)
\end{eqnarray*}
for every $y>0$ and $s>0$. As a result, one can complete the result by passing through the differentiation into the integral.
\end{proof}

Let $A_{T}$ and $A_{P}$ be the infinitesimal generators of $\{T_{t}\}_{t\geq0}$ and $\{P_{t}\}_{t\geq0}$ respectively. 
One can see that the generators have the formal relation $A_{P}=-(-A_{T})^{\frac{1}{2}}$. 
We set $R_{0}=\{f,A_{T}(f)\in\dom(A_{T})\}$ and $R_{n}=\cap_{k=1}^{n}\dom(A_{P}^{k})$. 
If $1\leq k\leq n$ and $f\in R_{n}$, then the $\frac{\partial^{k}}{\partial y^{k}}u_{f}$ belongs to the space $R_{n-k}$. 
Since $\{T_{t}\}_{t\geq0}$ and $\{P_{t}\}_{t\geq0}$ are Feller semigroups, the space $R_{n}$ is contained in $C_{0}(\cS)$ for every $n\geq0$, which implies that  $R_{n}$ is dense in $L^{p}$ for every $p\geq 1$ and every $n\geq 0$. 
This observation and the density argument enable us to restrict our attention to $C_{0}(\cS)$ in what follows.
We refer the reader to \cite[p.29]{Va80} and \cite[Chap IV \S10,\S11]{Yo71} for further discussion.

We review the celebrated inequalities for general semigroups by E. M. Stein in his 1970 monograph \cite{St70t}. 
The first inequality is the maximal ergodic theorem that plays a fundamental role in the proof of Theorem \ref{lem1}. 
In \cite{St70t}, the author gives two different proofs. 
One is to use the Hopf-Dunford-Schwartz ergodic theorem with an interpolation argument and the other is to utilize the martingale inequalities via a result of Rota \cite{Ro62} which sheds new light on the link between semigroups and martingales.  
For the completeness of the paper, we provide the continuous martingale version of the second proof, which can be found in I. Shigekawa \cite{Sh02}.

\begin{proposition}[Maximal ergodic theorem]\label{prop:stmax}
The harmonic extension of $f$ satisfies the following maximal inequality:
\begin{eqnarray*}
\norm{\sup_{y>0}|u_{f}(\cdot,y)|}_{p}\leq \frac{p}{p-1}\norm{f}_{p}
\end{eqnarray*}
for all $f\in L^{p}$ and for all $1<p\leq\infty$.  When $p=\infty$, we interpret the constant to be $1$. 
\end{proposition}
\begin{proof}
We prove the result for general symmetric Markovian semigroup $\{Q_{t}\}_{t\geq0}$. 
Let  $\{X_{t}\}_{t\geq0}$ be the stochastic process corresponding to $\{Q_{t}\}_{t\geq0}$, that is, $Q_{t}f(x)=\E^{x}[f(X_{t})]$ for any $f\in L^{p}$. We assume $1<p<\infty$ as $p=\infty$ is trivial. 
Let $T>0$ be fixed and $\{\cF_{t}:t\geq0\}$ the natural filtration of $\{X_t\}$. 
Using the Markov property and the semigroup property, we have
\begin{eqnarray*}
Q_{2(T-t)}f(X_{T})
&=&Q_{T-t}(Q_{T-t}f)(X_{T})\\
&=&\E^{X_{T}}[Q_{T-t}f(X_{T-t})]\\
&=&\E^{x}[Q_{T-t}f(X_{2T-t})|\cF_{T}].
\end{eqnarray*}
It then follows from Jensen's inequality that 
\begin{eqnarray*}
\sup_{0\leq t\leq T}\abs{Q_{2(T-t)}f(X_{T})}^{p}\leq \E^{x}[\sup_{0\leq t\leq T}\abs{Q_{T-t}f(X_{2T-t})}^p|\cF_{T}]
\end{eqnarray*}
and in turn by taking expectation on both sides and integrating over $\cS$ with respect to $dx$ that
\begin{eqnarray}\label{maxeq1}
\int_{\cS}\E^{x}[\sup_{0\leq t\leq T}\abs{Q_{2(T-t)}f(X_{T})}^{p}]dx
&\leq&\int_{\cS}\E^{x}[\sup_{0\leq t\leq T}\abs{Q_{T-t}f(X_{2T-t})}^{p}]dx\nonumber\\
&=&\int_{\cS}\E^{x}[\sup_{0\leq t\leq T}\abs{Q_{T-t}f(X_{t})}^{p}]dx.
\end{eqnarray}
We have used the reversibility of the process $X_{t}$ in the last line. 
Note that the process $Q_{T-t}f(X_{t}^{x})$ is a martingale. This is because 
\begin{eqnarray*}
Q_{T-t}f(X_{t})=\E^{x}[f(X_{T})|\cF_{t}].
\end{eqnarray*}
Then the Doob's maximal inequality yields that
\begin{eqnarray}\label{maxeq2}
\E^{x}[\sup_{0\leq t\leq T}\abs{Q_{T-t}f(X_{T})}^{p}]\leq\Paren{\frac{p}{p-1}}^{p}\E^{x}[|f(X_{T})|^{p}].
\end{eqnarray}
It follows from the self-adjointness and the invariant property $Q_{t}1=1$ that
\begin{eqnarray*}
\int_{\cS}\E^{x}[g(X_{T})]dx=\int_{\cS}Q_{T}g(x)dx=\int_{\cS}g(x)dx
\end{eqnarray*}
for any bounded measurable function $g$.
Applying this to \eqref{maxeq1} and \eqref{maxeq2} yields 
\begin{eqnarray*}
\norm{\sup_{0\leq t\leq T}|Q_{2(T-t)}f(x)|}_{p}
&\leq&\Paren{\int_{\cS}\E^{x}[\sup_{0\leq t\leq T}\abs{Q_{T-t}f(X_{t})}^{p}]dx}^{\frac{1}{p}}\\
&\leq&\frac{p}{p-1}\Paren{\int_{\cS}\E^{x}[|f(X_{T})|^{p}]dx}^{\frac{1}{p}}\\
&=&\frac{p}{p-1}\norm{f}_{p}.
\end{eqnarray*}
Since the RHS does not depend on $T$, we let $T\to\infty$ to complete the proof. 
\end{proof}

We end the subsection by stating the Littlewood-Paley inequality for symmetric Markovian semigroups in \cite{St70t} that will be a key estimate in the proof of Theorem \ref{thm1}. 
For a function $f\in L^{p}$, we define the Littlewood-Paley $g_k$--function by
\begin{eqnarray*}
g_{k}(f)(x)=\Paren{\int_{0}^{\infty}y^{2k-1}\Abs{\frac{\partial^{k}u_{f}}{\partial y^{k}}(x,y)}^{2}dt}^{1/2}, 
\end{eqnarray*}
for each $k\geq1$.
 
\begin{proposition}\label{prop:gfunc}
Let $1<p<\infty$ and $k\geq1$.
If $f\in L^{p}$, then the Littlewood-Paley function $g_{k}(f)$ is in $L^{p}$ as well and satisfies the inequality
\begin{eqnarray*}
\norm{g_{k}(f)}_{p}\leq C_{p,k}\norm{f}_{p}
\end{eqnarray*}
for some constant $C_{p,k}$ depending only on $p$ and $k$.
\end{proposition}
In the following, we only use the Littlewood-Paley $g$--function for $k=1$. 
We refer the reader to \cite[p.111, p.120]{St70s} for the detail.

\subsection{On stochastic processes}\label{sec:stocproc}
We define a stochastic process on the space $\cS\times\R$ associated to the strongly continuous symmetric Markovian semigroup $\{T_{t}\}_{t\geq0}$ of dimension $d$ as defined above. 
Let $H_{t}$ be the heat semigroup on $\R$ defined by
\begin{eqnarray*}
H_{t}f(x)=\frac{1}{(2\pi t)^{1/2}}\int_{\R}e^{-\frac{(x-y)^{2}}{2t}}f(y)dy.
\end{eqnarray*}
Given the product semigroup $\{T_{t}\times H_{t}\}_{t\geq0}$, one can construct the corresponding stochastic process $Z_{t}=(X_{t},Y_{t})\in\cS\times\R$ on the probability space $(\Omega,\sF,\P)$, whose paths are right-continuous with left limits and whose components, $X_{t}$ and $Y_{t}$, are independent each other. 
For example, if we restrict our attention to the standard heat semigroup on $\R^{n}$, the $(n+1)$-dimensional Brownian motion is the corresponding stochastic process.
From this point of view, the stochastic process $(Z_{t})_{t\geq0}$ is an analogue of a Brownian motion in $\cS\times \R$.
In addition, we assume that the stochastic process is to be killed when it leaves the upper half space $\cS\times [0,\infty)$. 
In other words, let $\tau:=\inf\{t\geq0:Y_{t}=0\}$ be the hitting time of $Y_{t}$ at 0 and consider the killed process $(Z_{t\wedge \tau})_{t\geq0}$ instead.

Let $s>0$ be fixed. 
We assume that the initial distribution of the stochastic process $(Z_{t})_{t\geq0}$ is given by $dx\otimes\delta_{s}$ where $\delta_{s}$ is the Dirac delta measure at fixed $s>0$. 
To put it another way, one can describe that the process $(Z_{t})_{t\geq0}$ starts at $(x_{0},s)\in\cS\times\R$ where $x_{0}$ is randomly chosen with respect to the measure $dx$. 
The probability and expectation of $Z_{t}$ with the initial distribution are denoted by $\cE$ and $\cP$ respectively. 
Explicitly, $\cE$ and $\cP$ are written as
\begin{eqnarray*}
\cE=\int_{\cS}\E^{(x,s)}dx,\quad \cP=\int_{\cS}\P^{(x,s)}dx.
\end{eqnarray*}
One thing to remark here is that it is possible for $\cP$ not to be a probability measure because the total mass $\int_{\cS}1dx$ do not have to be 1. 
However, as explained in \cite{Va80}, all the result from probability theory connected with this context remain valid.

\begin{lemma}\label{lem:ztfor}
\begin{enumerate}[$(i)$]
\item
For any function $h\in L^{1}(\cS)$, we have
\begin{eqnarray*}
\cE[h(X_{\tau})]=\int_{\cS}h(x)dx.
\end{eqnarray*}
\item For a Borel measurable function $f$ on $\cS\times\R$, we have the Green function formula for $Z_{t}$:
\begin{eqnarray}\label{eq:ztfor2}
\cE[\int_{0}^{\tau}f(Z_{t})dt]=2\int_{0}^{\infty}\int_{\cS}(y\wedge s)f(x,y)dxdy.
\end{eqnarray}
\end{enumerate}
\end{lemma}
\begin{proof}
We see by the facts $P_{t}1=1$ and $\langle P_{s}f, g\rangle=\langle f, P_{s}g\rangle$ that
\begin{eqnarray*}
\cE[h(X_{\tau})]
&=&\int_{\cS}\E^{(x,s)}h(X_{\tau})dx
=\int_{\cS}P_{s}h(x)dx\\
&=&\langle P_{s}h, 1\rangle
=\langle h, P_{s}1\rangle
=\int_{\cS}h(x)dx.
\end{eqnarray*}
For the proof of (ii), we refer the reader to \cite[Proposition 3.1]{Va80}. 
\end{proof}

We now proceed to define the stochastic integrals with respect to $\cP$. Consider a stochastic process $(A_{t})_{t\geq0}$ satisfying the following conditions: 
\begin{enumerate}[(i)]
\item the map $A:\Omega\times[0,\infty)\to\R$ is jointly measurable,
\item $A_{t}\in\sF_{t}$ for every $t\in[0,\infty)$,
\item $\cE[\int_{0}^{\infty}\abs{A_{t}}^{2}dt]<\infty.$
\end{enumerate}
We denote by $(A_{t})_{t\geq0}\in L^{2}(\Omega,\cP)$. Given such a process $(A_{t})_{t\geq0}$, we define the stochastic integral against $(Y_{t})_{t\geq0}$
\begin{eqnarray*}
I(A)_{t}:=\int_{0}^{t}A_{s}dY_{s}
\end{eqnarray*}
in a canonical way. 

If the ``probability'' $\cP$ is finite, then one can define $I(A)_{t}$ as a $L^{2}$-limit of martingale transforms with the aid of the It\^o's isometry.
In the case where $\cP$ is infinite, we decompose the Radon measure $dx$ into a countable family of finite measures $dx_{n}$, define the stochastic integral for each finite measure $dx_{n}$ and let the integral for whole space as the sum of integrals. 
The finiteness of the sum is assured by the third assumption of $(A_{t})_{t\geq0}$. 
We refer the reader to \cite[pp.37-38]{Va80} where the detail of the construction is presented. 

We are ready to state the \emph{the projection lemma}, which provides an effective tool to handle stochastic integrals in the proof of Theorem \ref{thm1}. 
This is an analogue of It\^o's formula for the $(d+1)$-dimensional Brownian motion.
The theorem says that the $(Y_{t})$-directional component of $u_{f}(Z_{t\wedge \tau})$ can be represented as a stochastic integral against $Y_{t}$. 
We will just state the theorem without proof because the proof is fairly lengthy. 
Instead we refer the reader to \cite[pp.50-59]{Va80} and other references therein. 

Let $V$ be the set of all stochastic processes in $L^{2}(\Omega,\cP)$ which is of the form $(I(A)_{t})_{t\geq0}$ as defined above. One sees easily that $V$ is a closed subspace. Let $\proj_{V}$ be the orthogonal projection from $L^{2}(\Omega,\cP)$ onto $V$. Then the projection theorem asserts as follows.
\begin{proposition}\label{prop:proj}
If $f\in R_{5}$, then 
\begin{eqnarray*}
\proj_{V}(u_{f}(Z_{t\wedge \tau})-u_{f}(Z_{0}))=\int_{0}^{t\wedge\tau}\frac{\partial u_{f}}{\partial y}(Z_{s})dY_{s}.
\end{eqnarray*}
\end{proposition}

\section{Proofs of the main results}
In this section, we provide the proofs of Theorem \ref{thm1} and Theorem \ref{lem1}. 
In the first place, we present the proof of Theorem \ref{lem1} which produces the Hardy-Littlewood-Sobolev inequality for the fractional square function $\cG_{\alpha}$. 

\begin{proof}[Proof of Theorem \ref{lem1}]
We begin by splitting $\cG_{\alpha}(f)^{2}$ into two parts
\begin{eqnarray*}
\cG_{\alpha}(f)(x)^{2}
&=&\int_{0}^{\infty}y^{2\alpha+1}\abs{\frac{\partial u_{f}}{\partial y}(x,y)}^{2}dy\\
&=&\int_{0}^{\delta}y^{2\alpha+1}\abs{\frac{\partial u_{f}}{\partial y}(x,y)}^{2}dy+
	\int_{\delta}^{\infty}y^{2\alpha+1}\abs{\frac{\partial u_{f}}{\partial y}(x,y)}^{2}dy.
\end{eqnarray*}
For the first integral, we apply Lemma \ref{lem:ptest} to obtain
\begin{eqnarray*}
\int_{0}^{\delta}y^{2\alpha+1}\abs{\frac{\partial u_{f}}{\partial y}(x,y)}^{2}dy
&\leq& c_{1}\int_{0}^{\delta}y^{2\alpha-1}\Abs{u_{|f|}(x,\frac{1}{\sqrt{2}}y)}^{2}dy\\
&\leq& C_{\alpha}\sup_{y>0}\Abs{u_{|f|}(x,y)}^{2}\delta^{2\alpha}.
\end{eqnarray*}
Likewise the other integral can be bounded using Lemma \ref{lem:sms} and Lemma \ref{lem:ptest}
\begin{eqnarray*}
\int_{\delta}^{\infty}y^{2\alpha+1}\abs{\frac{\partial u_{f}}{\partial y}(x,y)}^{2}dy
\leq c_{1}\int_{\delta}^{\infty}y^{2\alpha-1}\abs{u_{|f|}(x,\frac{1}{\sqrt{2}}y)}^{2}dy
\leq C_{\alpha}\norm{f}^{2}_{p}\delta^{2(\alpha-\frac{d}{p})}.
\end{eqnarray*}
Consequently, we see
\begin{eqnarray*}
\cG_{\alpha}(f)(x)
\leq C_{\alpha,p,d}\paren{
\sup_{y>0}\abs{u_{|f|}(x,y)}\delta^{\alpha}+\norm{f}_{p}\delta^{\alpha-\frac{d}{p}}
}
\end{eqnarray*}
for some constant $C_{\alpha,p,d}$.
Optimizing the RHS in $\delta$ yields
\begin{eqnarray*}
\cG_{\alpha}(f)(x)\leq C_{\alpha,p,d}\paren{\sup_{y>0}\abs{u_{|f|}(x,y)}}^{1-\frac{\alpha p}{d}}\norm{f}_{p}^{\frac{\alpha p}{d}}.
\end{eqnarray*}
We observe by Proposition \ref{prop:stmax} that
\begin{eqnarray*}
\norm{\paren{\sup_{y>0}\abs{u_{|f|}(x,y)}}^{1-\frac{\alpha p}{d}}}_{q}
=\norm{\sup_{y>0}\abs{u_{|f|}(x,y)}}_{p}^{\frac{p}{q}}
\leq C_{p}\norm{f}_{p}^{\frac{p}{q}}
\end{eqnarray*}
since $1-\frac{\alpha p}{d}=\frac{p}{q}$.
As a result, we acquire 
\begin{eqnarray*}
\norm{\cG_{\alpha}(f)}_{q} 
&\leq& C_{\alpha,p,d}\norm{\paren{\sup_{y>0}\abs{u_{|f|}(x,y)}}^{p/q}}_{q}\norm{f}_{p}^{1-p/q}\\
&=& C_{\alpha,p,d}\norm{\paren{\sup_{y>0}\abs{u_{|f|}(x,y)}}}^{p/q}_{p}\norm{f}_{p}^{1-p/q}\\
&\leq&C_{\alpha,p,d}\norm{f}_{p},
\end{eqnarray*}
which finishes the proof.
\end{proof}

\begin{proof}[Proof of Theorem \ref{thm1}]
In the first place, we want to show
\begin{eqnarray}\label{eq:step1}
\cE[\int_{0}^{\tau}Y_{t}^{\alpha}\Abs{\frac{\partial u_{f}}{\partial y}(Z_{t})}\Abs{\frac{\partial u_{h}}{\partial y}(Z_{t})}dt]
\leq C_{\alpha,p,d}\norm{f}_{p}\norm{h}_{q'}.
\end{eqnarray}
Applying the Green formula \eqref{eq:ztfor2}, we see
\begin{eqnarray*}
\cE[\int_{0}^{\tau}Y_{t}^{\alpha}\Abs{\frac{\partial u_{f}}{\partial y}(Z_{t})}\Abs{\frac{\partial u_{h}}{\partial y}(Z_{t})}dt]
=2\int_{\cS}\int_{0}^{\infty}(y\wedge s)y^{\alpha}\Abs{\frac{\partial u_{f}}{\partial y}}\Abs{\frac{\partial u_{h}}{\partial y}}dydx.
\end{eqnarray*}
Then H\"older inequality tells us that
\begin{eqnarray*}
\int_{\cS}\int_{0}^{\infty}(y\wedge s)y^{\alpha}\Abs{\frac{\partial u_{f}}{\partial y}}\Abs{\frac{\partial u_{h}}{\partial y}}dydx
&\leq& \int_{\cS}\int_{0}^{\infty}y^{\alpha+\frac{1}{2}}\Abs{\frac{\partial u_{f}}{\partial y}}y^{\frac{1}{2}}\Abs{\frac{\partial u_{h}}{\partial y}}dydx\\
&\leq& \int_{\cS}\cG_{\alpha}(f) g_{1}(h)dx\\
&\leq& \norm{\cG_{\alpha}(f)}_{q}\norm{g_{1}(h)}_{q'}.
\end{eqnarray*}
Thus the claim follows from Proposition \ref{prop:gfunc} and Theorem \ref{lem1}. 

Second, we will prove
\begin{eqnarray*}
\brk{\cT_{\alpha}^{s}(f),h}=\cE[\int_{0}^{\tau}Y_{t}^{\alpha}\frac{\partial u_{f}}{\partial y}(Z_{t})\frac{\partial u_{h}}{\partial y}(Z_{t})dt]
\end{eqnarray*}
We truncate $\sT_{\alpha}^{s}$ in a way that for $N>0$
\begin{eqnarray*}
\sT_{\alpha}^{s, N}(f)(x)=\cE[\int_{0}^{\tau}(Y_{t}^{\alpha}\wedge N)\frac{\partial u_{f}}{\partial y}(Z_{t})dY_{t}|X_{\tau}=x].
\end{eqnarray*}
From Lemma \ref{lem:ztfor} and the properties of conditional expectation, we then see that
\begin{eqnarray*}
\brk{\sT_{\alpha}^{s, N}(f),h}
&=&\cE[\sT_{\alpha}^{s, N}(f)(X_{\tau})h(X_{\tau})]\\
&=&\cE[\cE[\int_{0}^{\tau}(Y_{t}^{\alpha}\wedge N)\frac{\partial u_{f}}{\partial y}(Z_{t})dY_{t}|X_{\tau}]h(X_{\tau})]\\
&=&\cE[\cE[h(X_{\tau})\int_{0}^{\tau}(Y_{t}^{\alpha}\wedge N)\frac{\partial u_{f}}{\partial y}(Z_{t})dY_{t}|X_{\tau}]]\\
&=&\cE [h(X_{\tau})\int_{0}^{\tau}(Y_{t}^{\alpha}\wedge N)\frac{\partial u_{f}}{\partial y}(Z_{t})dY_{t}].
\end{eqnarray*}
Note that the stochastic integral $\int_{0}^{\tau}(Y_{t}^{\alpha}\wedge N)\frac{\partial u_{f}}{\partial y}(Z_{t})dY_{t}$ belongs to $L^{2}(\Omega,\cP)$. Indeed, it follows from the Green formula \eqref{eq:ztfor2} that
\begin{eqnarray*}
\cE[\int_{0}^{\tau}(Y_{t}^{2\alpha}\wedge N^{2})\Abs{\frac{\partial u_{f}}{\partial y}(Z_{t})}^{2}dt]
&\leq& 2N^{2}\int_{0}^{\infty}\int_{\cS}y\Abs{\frac{\partial u_{f}}{\partial y}(x,y)}^{2}dxdy\\
&=& 2N^{2}\norm{g_{1}(f)}_{2}^{2}\\
&\leq& cN^{2}\norm{f}_{2}^{2}<\infty.
\end{eqnarray*}
Furthermore the integral $\int_{0}^{\tau}(Y_{t}^{\alpha}\wedge N)\frac{\partial u_{f}}{\partial y}(Z_{t})dY_{t}\in V$, where $V$ is the closed subspace of $L^{2}(\Omega,\cP)$ of all stochastic integral with respect to $(Y_{t})_{t\geq0}$. Thus the projection lemma(Proposition \ref{prop:proj}) yields that
\begin{eqnarray*}
\brk{\sT_{\alpha}^{s, N}(f),g}
&=&\cE [\Paren{\int_{0}^{\tau}\frac{\partial u_{h}}{\partial y}(Z_{t})dY_{t}}\Paren{\int_{0}^{\tau}(Y_{t}^{\alpha}\wedge N)\frac{\partial u_{f}}{\partial y}(Z_{t})dY_{t}}]\\
&=&\cE [\int_{0}^{\tau}(Y_{t}^{\alpha}\wedge N)\frac{\partial u_{f}}{\partial y}(Z_{t})\frac{\partial u_{h}}{\partial y}(Z_{t})dt].
\end{eqnarray*}
By letting $N\to\infty$ with the help of \eqref{eq:step1} and the dominated convergence theorem, we conclude the second claim.

What is left is to show that $\cT_{\alpha}^{s}f$ converges to $c_{\alpha}\cI_{\alpha}(f)$ as $s$ tends to $\infty$ in distribution sense. 
Note that we see by \eqref{eq:thm1} and the Green function formula \eqref{eq:ztfor2} that 
\begin{eqnarray*}
\langle \cT^{s}_{\alpha}f,h \rangle
=2\int_{0}^{\infty}\int_{\cS}(y\wedge s)y^{\alpha}\frac{\partial u_{f}}{\partial y}(x,y)\frac{\partial u_{h}}{\partial y}(x,y)dxdy.
\end{eqnarray*}
Thus it is enough to show 
\begin{eqnarray*}
\langle \cI_{\alpha}f,h \rangle
=C_{\alpha}\int_{0}^{\infty}\int_{\cS}y^{\alpha+1}\frac{\partial u_{f}}{\partial y}(x,y)\frac{\partial u_{h}}{\partial y}(x,y)dxdy.
\end{eqnarray*}
To see this, we exploit the spectral representation for the Poisson semigroup. Since $f$ and $g$ are in $L^{2}$, we write
\begin{eqnarray*}
\int_{\cS}\frac{\partial u_{f}}{\partial y}(x,y)\frac{\partial u_{h}}{\partial y}(x,y)dx
&=&\langle \frac{\partial u_{f}}{\partial y}(\cdot,y),\frac{\partial u_{h}}{\partial y}(\cdot,y)\rangle\\
&=&\langle \int_{0}^{\infty}\lambda^{1/2}e^{-\lambda^{1/2}y}dE_{\lambda}f,\int_{0}^{\infty}\lambda^{1/2}e^{-\lambda^{1/2}y}dE_{\lambda}h\rangle\\
&=&\int_{0}^{\infty}\lambda e^{-2\lambda^{1/2}y}d\langle E_{\lambda}f,E_{\lambda}h\rangle.
\end{eqnarray*}
It then follows from Fubini's theorem that
\begin{eqnarray*}
\int_{0}^{\infty}\int_{\cS}y^{\alpha+1}\frac{\partial u_{f}}{\partial y}(x,y)\frac{\partial u_{h}}{\partial y}(x,y)dxdy
&=&\int_{0}^{\infty}y^{\alpha+1}\langle \frac{\partial u_{f}}{\partial y}(\cdot,y),\frac{\partial u_{h}}{\partial y}(\cdot,y)\rangle dy\\
&=&\int_{0}^{\infty}y^{\alpha+1}\paren{\int_{0}^{\infty}\lambda e^{-2\lambda^{1/2}y}d\langle E_{\lambda}f,E_{\lambda}h\rangle} dy\\
&=&\int_{0}^{\infty}\lambda \paren{\int_{0}^{\infty}y^{\alpha+1}e^{-2\lambda^{1/2}y} dy}d\langle E_{\lambda}f,E_{\lambda}h\rangle\\
&=&\frac{\Gamma(\alpha+2)}{2^{\alpha+2}}\int_{0}^{\infty}\lambda^{-\alpha/2} d\langle E_{\lambda}f,E_{\lambda}h\rangle\\
&=&C_{\alpha}\langle I_{\alpha}f,h \rangle,
\end{eqnarray*}
which completes the proof.
\end{proof}

\subsection*{Acknowledgement}
I would like to thank Professor Rodrigo Ba\~nuelos, my academic advisor, for suggesting this problem and for his invaluable help and encouragement while writing this paper.


\begin{thebibliography}{99}


\bibitem{AB13}
D.~Applebaum and R.~Ba\~nuelos,
\newblock Probabilistic Approach to Fractional Integrals and the Hardy-Littlewood-Sobolev Inequality,
\newblock {\em Analytic Methods in Interdisciplinary Applications},
\newblock {Springer Proc. Math. Stat.} \textbf{116},
17--40 (2015)

\bibitem{Ba86}
R.~Ba{\~n}uelos,
\newblock Martingale transforms and related singular integrals,
\newblock {\em Transactions of the American Mathematical Society} \textbf{293} no.~2,
547--563 (1986)

\bibitem{BanMen} 
R. Ba\~nuelos and P. J. M\'endez-Hern\'andez, 
{Space-time Brownian motion and the Beurling-Ahlfors transform}, 
{\em Indiana Univ. Math. J.} \textbf{52} no.~4, 
981--990 (2003)

\bibitem{Ba10f}
R.~Ba\~nuelos,
\newblock The foundational inequalities of D. L. Burkholder and some of their ramifications,
\newblock {\em Illinois Journal of Mathematics} \textbf{54} no.~3,
789--868 (2010)

\bibitem{Ba13s}
R.~Ba{\~n}uelos and A.~Os\c ekowski,
\newblock Sharp martingale inequalities and applications to Riesz transforms on manifolds,
\newblock {\em Lie Groups and Gauss Space}, 
submitted (2013)

\bibitem{Ba95}
R.~F. Bass,
\newblock {\em Probabilistic techniques in analysis},
\newblock Springer (1995)

\bibitem{Bo12}
S.~Bochner,
\newblock {\em Harmonic Analysis and the Theory of Probability},
\newblock Dover (2012)

\bibitem{Bu84}
D.~L. Burkholder,
\newblock Boundary value problems and sharp inequalities for martingale transforms,
\newblock {\em The Annals of Probability}, 
\textbf{12}, 647--702 (1984)

\bibitem{CL90}
E.~A.~Carlen and M.~Loss,
\newblock Extremals of functionals with competing symmetries,
\newblock {\em J. Funct. Anal.} \textbf{88},
437--456 (1990)

\bibitem{Da90}
E.~B. Davies,
\newblock {\em Heat kernels and spectral theory}, 
\newblock Cambridge University Press (1990)

\bibitem{FL12a}
R.~L. Frank and E.~H. Lieb,
\newblock A new, rearrangement-free proof of the Sharp Hardy-Littlewood-Sobolev Inequality,
\newblock {\em Spectral Theory, Function Spaces and Inequalities Oper. Theory: Adv. Appl.} \textbf{219},
55--67 (2012)

\bibitem{FL12}
R.~L. Frank and E.~H. Lieb,
\newblock Sharp constants in several inequalities on the Heisenberg group,
\newblock {\em Ann. of Math.} \textbf{176},
349--381 (2012)

\bibitem{FOT10}
M.~Fukushima, Y.~Oshima and M.~Takeda,
\newblock {\em Dirichlet forms and symmetric Markov processes},
\newblock Walter de Gruyter (2010)

\bibitem{Ge10}
S.~Geiss, S.~Montgomery-Smith and E.~Saksman,
\newblock On singular integral and martingale transforms,
\newblock {\em Transactions of the American Mathematical Society} \textbf{362} no.~2,
553--575 (2010)

\bibitem{GT01}
D.~Gilbarg and N.~S. Trudinger,
\newblock {\em Elliptic partial differential equations of second order},
\newblock Springer (2001)

\bibitem{GV79}
R.~F. Gundy and N.~T. Varopoulos,
\newblock Les transformations de riesz et les int{\'e}grales stochastiques,
\newblock {\em CR Acad. Sci. Paris S{\'e}r. A-B} \textbf{289} no.~1,
A13--A16, (1979)

\bibitem{HL28}
G.~H.~Hardy and J.~E.~Littlewood,
\newblock Some properties of fractional integrals (1),
\newblock {\em Math. Zeitschr.} \textbf{27},
565--606 (1928)

\bibitem{HL30}
G.~H.~Hardy and J.~E.~Littlewood,
\newblock On certain inequalities connected with the calculus of variations,
\newblock {\em J. London Math. Soc.} \textbf{5},
34--39 (1930)

\bibitem{He72}
L.~I. Hedberg,
\newblock On certain convolution inequalities,
\newblock {\em Proceedings of the American Mathematical Society} \textbf{36} no.~2,
505--510 (1972)

\bibitem{Li08}
X.~D.~Li,
\newblock Martingale transforms and $L^{p}$-norm estimates of Riesz transforms on complete Riemannian manifolds,
\newblock {\em Probability Theory and Related Fields} \textbf{141},
247--281 (2008)

\bibitem{Li83}
E.~H. Lieb,
\newblock Sharp constants in the Hardy-Littlewood-Sobolev and related inequalities,
\newblock {\em Annals of Mathematics} \textbf{118} no.~2,
349--374 (1983)

\bibitem{Os12}
A.~Os\c ekowski,
\newblock {\em Sharp martingale and semimartingale inequalities},
\newblock Springer (2012)

\bibitem{Os3} 
A.~Os\c ekowski,
{Sharp logarithmic inequalities for Riesz transforms}, 
{\em Journal of Functional Analysis} \textbf{263},
89--108 (2012)

\bibitem{RY99}
D.~Revuz and M.~Yor,
\newblock {\em Continuous Martingales and Brownian Motion},
\newblock Springer (1999)

\bibitem{Ro62}
G.~C.~Rota,
\newblock An ``alternierende verfahren'' for general positive operators,
\newblock {\em Bulletin of the American Mathematical Society} \textbf{68} no.~2,
95--102 (1962)

\bibitem{Sh02}
I.~Shigekawa et~al,
\newblock Littlewood-Paley inequality for a diffusion satisfying the logarithmic Sobolev inequality and for the Brownian motion on a Riemannian manifold with boundary,
\newblock {\em Osaka J. Math} \textbf{39},
897--930 (2002)

\bibitem{So38}
S.~L.~Sobolev,
\newblock On a theorem of functional analysis,
\newblock {\em Mat. Sbornik} \textbf{4},
471--497 (1938)

\bibitem{St70s}
E.~M.~Stein,
\newblock {\em Singular integrals and differentiability properties of functions},
\newblock Princeton university press (1970)

\bibitem{St70t}
E.~M.~Stein,
\newblock {\em Topics in harmonic analysis related to the Littlewood-Paley theory},
\newblock Princeton University Press (1970)

\bibitem{Va85}
N.~Th.~Varopoulos,
\newblock Hardy-Littlewood theory for semigroups,
\newblock {\em Journal of Functional Analysis} \textbf{63} no.~2,
240--260 (1985)

\bibitem{Va80}
N.~Th.~Varopoulos,
\newblock Aspects of probabilistic Littlewood-Paley theory,
\newblock {\em Journal of Functional Analysis} \textbf{38} no.~1,
25--60 (1980)

\bibitem{SalVarCou}  
N. Th. Varopoulos, L. Saloff-Coste and A.T. Coulhon, 
{\em Analysis and geometry on groups}, 
Cambridge University Press (1993)

\bibitem{VN04}
A.~Volberg and F.~Nazarov,
\newblock Heating of the Ahlfors--Beurling operator and estimates of its norm,
\newblock {\em St. Petersburg Mathematical Journal} \textbf{15} no.~4,
563--573 (2004)

\bibitem{Yo71}
K.~Yosida,
\newblock {\em Functional analysis}, 
\newblock {Spring-Verlag, New York/Berlin} (1980)

\end{thebibliography}
\end{document}